\documentclass[british, 12pt, reqno]{amsart}
\usepackage[T1]{fontenc}
\usepackage[latin9]{inputenc}
\usepackage{amsmath, amsthm, amssymb, stmaryrd}

\usepackage{geometry}
\usepackage{amsmath}
\usepackage{amsthm}
\usepackage{amssymb}
\usepackage{xcolor}

\usepackage{enumerate}
\usepackage{hyperref}
\usepackage{graphicx}

\makeatletter
\newtheorem{thm}{\protect\theoremname}[section]
\theoremstyle{plain}
\newtheorem{cor}[thm]{\protect\corollaryname}
\ifx\proof\undefined
\newenvironment{proof}[1][\protect\proofname]{\par
	\normalfont\topsep6\p@\@plus6\p@\relax
	\trivlist
	\itemindent\parindent
	\item[\hskip\labelsep\scshape #1]\ignorespaces
}{%
	\endtrivlist\@endpefalse
}
\providecommand{\proofname}{Proof}
\fi
\theoremstyle{remark}

\theoremstyle{plain}
\newtheorem{lem}[thm]{\protect\lemmaname}

\makeatother

\usepackage{babel}
\providecommand{\corollaryname}{Corollary}
\providecommand{\lemmaname}{Lemma}
\providecommand{\remarkname}{Remark}
\providecommand{\theoremname}{Theorem}

\numberwithin{equation}{section}
\numberwithin{figure}{section}
\theoremstyle{plain}
\theoremstyle{plain}

\def\1{{\textcolor{red} {1}}}
\def\d1{{\textcolor{red} {d-1}}}

\def \bN {\mathbb N}

\def \le {\leqslant}

\def \ge {\geqslant}

\def \d {{\mathrm{d}}}

\def \ds1 {\mathds{1}}

\def \eps {{\varepsilon}}

\def \lam {{\lambda}}

\begin{document}

\author{Sam Chow \and Owen Jones}

\address{Mathematics Institute, Zeeman Building, University of Warwick, Coventry CV4 7AL, United Kingdom}
\email{sam.chow@warwick.ac.uk}

\address{Mathematics Institute, Zeeman Building, University of Warwick, Coventry CV4 7AL, United Kingdom}
\email{Owen.S.Jones@warwick.ac.uk}

\title{On the variance of the Fibonacci partition function}

\maketitle
\begin{abstract}
We determine the order of magnitude of the variance of the Fibonacci partition function. The answer is different to the most naive guess. The proof involves a diophantine system and an inhomogeneous linear recurrence.
\end{abstract}

\section{Introduction}

The \emph{Fibonacci partition function} $R(n)$ counts solutions to
\[
x_1 + \cdots + x_s = n,
\]
where $x_1 < \cdots < x_s$ are Fibonacci numbers. Its values form the sequence OEIS A000119. By Zeckendorf's theorem, which uniquely expresses a positive integer as a sum of non-consecutive Fibonacci numbers, we have
\[
R(n) \ge 1 \qquad (n \in \bN).
\]
Observe that $R(0) = 1$, since the empty sum vanishes, and that 
\[
R(n) = 0 \qquad (n < 0).
\]
Several recursive formulas have been provided over the years \cite{Car1968, Rob1996, Wei2016}. Most recently, Chow and Slattery \cite{CS2021} furnished a fast, practical algorithm to compute $R(n)$, based on the Zeckendorf expansion.

The summatory function
\[
A(H) = \sum_{n=0}^H R(n)
\]
was shown in \cite{CS2021} to have order of magnitude $H^\lam$, where
\[
\varphi = \frac{1+\sqrt5}{2},
\qquad
\lam = \frac{\log 2}
{\log \varphi}
\approx 1.44.
\]
Thus, the average $(H+1)^{-1} \sum_{n=0}^H R(n)$ behaves fairly nicely, having order of magnitude $H^{\lam - 1}$. Chow and Slattery \cite{CS2021} demonstrated that $\displaystyle \lim_{H \to \infty} H^{-\lam} A(H)$ does not exist. There it was also found that
$ \displaystyle
\lim_{m \to \infty} F_m^{-\lam} A(F_m)
$
does exist, and such limits were recently investigated in far more generality by Zhou \cite{Zhou}.

The Fibonacci partition function is highly erratic. A simple exercise reveals, perhaps surprisingly at first, that
\[
R(F_m - 1) = 1
\qquad (m \ge 2).
\]
At the other extreme, Stockmeyer \cite{Sto2008} showed that
\[
R(F_m^2 - 1) = F_m \qquad 
(m \ge 2),
\]
and moreover
\[
R(n) \le \sqrt{n+1}
\]
with equality if and only if we have $n = F_m^2 - 1$ for some $m \ge 2$.

In this article, we quantify the fluctuations of $R(n)$ by estimating the second moment
\[
V(H) := \sum_{n = 0}^H R(n)^2.
\]
By Cauchy--Schwarz, we have
\[
V(H) \ge H^{-1} A(H)^2 \gg H^{2 \lam - 1}.
\]
If the function $R(n)$ were not too erratic, then we would have 
$V(H) \asymp H^{2 \lam - 1}$. One might naively guess this if computers were less powerful, see Figure \ref{Vgraph1}.

\begin{figure}[!htb] 
\centering
\includegraphics[width=8cm]{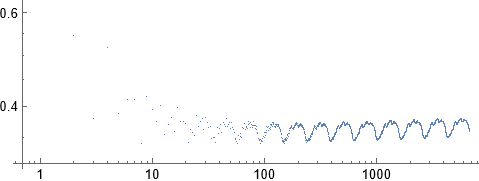}
\caption{$V(H)H^{1-2\lambda}$ against $H$ for $1 \le H \le 6765$}
\label{Vgraph1}
\end{figure}

However, the growth of $V(H)H^{1-2\lambda}$ becomes clearer with more data, see Figure \ref{Vgraph2}.

\begin{figure}[!htb] 
\centering
\includegraphics[width=8cm]{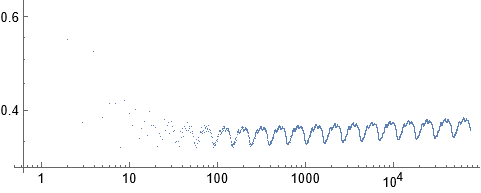}
\caption{$V(H)H^{1-2\lambda}$ against $H$ for $1 \le H \le 75025$}
\label{Vgraph2}
\end{figure}

We find that $V(H)$ grows like a slightly larger power of $H$.

\begin{thm} \label{MainThm}
Let $\lam_1 \approx 2.48$ be the greatest root of the polynomial 
\[
x^3 - 2 x^2 - 2 x + 2.
\]
Then, for $H \in \bN$, we have
\[
V(H) \asymp H^{(\log \lam_1)/ \log \varphi}.
\]
\end{thm}

\noindent The exponent
\[
\frac{\log \lam_1}{\log \varphi} \approx 1.89
\]
slightly exceeds the Cauchy--Schwarz exponent
\[
2 \lam - 1 \approx 1.88.
\]

\subsection*{Notation}
We employ the usual Bachman--Landau and Vinogradov asymptotic notations. For positive-valued functions $f$ and $g$, we write:
\begin{itemize}
\item $f \ll g$, if $|f| \le C|g|$ pointwise, for some constant $C$,
\item $f \asymp g$, if $f \ll g \ll f$,
\item $f \sim g$, if $f/g \to 1$ in some specified limit.
\end{itemize}

\subsection*{Organisation and methods}

In \S \ref{ineqs}, we interpret $V(F_m) - V(F_{m-1})$ as the number of solutions to a system comprising a diophantine equation and two inequalities. A case analysis then leads to a recurrence for $V(F_m)$. We then solve the inhomogeneous linear recurrence in \S \ref{solve}, delivering an exact formula and hence an asymptotic formula for $V(F_m)$. From there, we deduce Theorem \ref{MainThm}.

\subsection*{Funding}
OJ was supported by a URSS bursary from the University of Warwick.

\subsection*{Rights}

For the purpose of open access, the authors have applied a Creative Commons Attribution (CC-BY) licence to any Author Accepted Manuscript version arising from this submission.

\section{A diophantine system}
\label{ineqs}

We start by analysing the variance along the Fibonacci sequence.

\begin{lem} For $m \ge 7$, we have
\begin{align*}
V(F_m) 
&= 2V(F_{m-1}) + 3V(F_{m-2})
- 4V(F_{m-3}) - 2V(F_{m-4}) + 2V(F_{m-5}) \\
&\qquad + 1 - 2 \lfloor m/2 \rfloor.
\end{align*}
\end{lem}

We will repeatedly use the fact that 
\[
F_2+F_3+\cdots+F_{m-2} = F_m - 2 < F_m
\]
throughout the proof.

\begin{proof}
Observe that $R(n)^2$ counts solutions to 
\[
x_1 + \cdots + x_s = y_1+ \cdots + y_t=n,
\]
where $x_1<\cdots<x_s$ and $y_1<\cdots<y_t$ are Fibonacci numbers. Therefore $V(F_m)-V(F_{m-1})$ counts solutions to
\begin{equation} \label{big}
F_{m-1} < x_1 + \cdots + x_s = y_1+ \cdots + y_t \le F_m 
\end{equation}
with the same conditions on the $x_i$ and $y_j$.
Note that $x_s$ and $y_t$ must both be at least $F_{m-2}$, since otherwise the sums would be too small. However, it cannot be that one of these is $F_{m-2}$ and the other is $F_m$, since then the sums could not be equal. It is also clear that neither of these can exceed $F_m$. There are thus five possibilities:
\begin{enumerate}
\item $x_s=y_t=F_m$
\item $x_s=y_t=F_{m-1}$
\item $x_s=y_t=F_{m-2}$
\item $\{x_s,y_t\}
=\{F_m,F_{m-1}\}$
\item $\{x_s,y_t\}
=\{F_{m-1},F_{m-2}\}$.
\end{enumerate}

\bigskip

Case $(1)$ can only happen if $s=t=1$, since otherwise the sums will be too big. Exactly one solution arises from this case.

In Case $(2)$, the problem is equivalent to counting solutions to
\[
0 < x_1 + \cdots + x_s = y_1+ \cdots + y_t \le F_{m-2}, 
\]
of which there are 
$V(F_{m-2})-V(0)=V(F_{m-2})-1$. As $s$ and $t$ are arbitrary, we have relabelled these subscripts here.

\bigskip

Case (3) requires significantly more thought. First, we subtract $F_{m-2}$ and relabel subscripts to see that we are counting solutions to 
\begin{equation} \label{case3}
F_{m-3} < x_1 + \cdots + x_s = y_1+ \cdots + y_t \le F_{m-1}
\end{equation} 
for which
\[
x_1 < \dots < x_s <F_{m-2}, 
\qquad
y_1 < \dots < y_t <F_{m-2}.
\]
Thus, we must subtract from 
$V(F_{m-1})-V(F_{m-3})$ the number of solutions to \eqref{case3} in which at least one of $x_s$ or $y_t$ is an element of $\{F_{m-2},F_{m-1}\}$. 
When $x_s=F_{m-1}$, the system (\ref{case3}) becomes
\[
F_{m-1}=y_1+\cdots+y_t,
\]
which has $R(F_{m-1})$ solutions. The case $y_t=F_{m-1}$ is identical, but we do not wish to double-count solutions with $x_s=y_t=F_{m-1}$, so there are $2R(F_{m-1})-1$ solutions to (\ref{case3}) in which at least one of $x_s$ or $y_t$ is $F_{m-1}$. 

Suppose instead that $x_s = F_{m-2}$ and $y_t \ne F_{m-1}$. Then $y_t \in \{F_{m-3},F_{m-2}\}$, for otherwise the sum of the $y_j$ would be below $F_{m-2}$. The case where $y_t=x_s=F_{m-2}$ yields $V(F_{m-3})$ solutions, as we can subtract $F_{m-2}$ from (\ref{case3}). Let $w_m$ count solutions to 
\begin{equation} 
\label{case3a}
F_{m-3} 
< x_1 + \cdots + x_{s-1} + F_{m-2} 
= y_1+ \cdots + y_{t-1} + F_{m-3} 
\le F_{m-1} 
\end{equation} 
for which
\begin{equation}
\label{Wcons}
x_1< \cdots <x_{s-1}<F_{m-2},
\qquad
y_1< \cdots <y_{t-1}<F_{m-3}.
\end{equation}
This is the system that arises upon specialising $x_s=F_{m-2}$ and $y_t=F_{m-3}$ in (\ref{case3}) so, flipping the roles of $x_s$ and $y_t$, we find that the total number of solutions in Case $(3)$ is  \begin{align}
&\notag
V(F_{m-1})-V(F_{m-3}) - 2R(F_{m-1})+1-V(F_{m-3})-2w_m \\
&\label{3total}
= V(F_{m-1})-2V(F_{m-3}) 
- 2R(F_{m-1})+1-2w_m.
\end{align}

We proceed to compute $w_m$.
Observe that in (\ref{case3a}) we must have 
\[
x_1+\cdots+x_{s-1}+F_{m-4}
= y_1+\cdots+y_{t-1},
\]
hence the right sum is at least 
$F_{m-4}$ and so $y_{t-1} \ge 
F_{m-5}$. The conditions on the $y_j$ mean that the only possibilities for $y_{t-1}$ are $F_{m-4}$ and $F_{m-5}$ (which are distinct, as $m \ge 7$). 

The number of solutions for which 
$y_{t-1}=F_{m-4}$ is the same as the number of solutions to 
\begin{equation} \label{case3ai}
x_1+\cdots+x_s=y_1+\cdots+y_t \le F_{m-3}
\end{equation} 
for which 
\[
x_1 < \dots < x_s,
\qquad
y_1 < \dots < y_t < F_{m-4}.
\]
The total number of solutions to (\ref{case3ai}) is $V(F_{m-3})$; however, we must discount the solutions for which 
$y_t \in \{ F_{m-3}, F_{m-4} \}$. It is easy to see that the number of solutions to (\ref{case3ai}) in which $y_t=F_{m-3}$ is $R(F_{m-3})$. Instead, replacing $y_t$ with $F_{m-4}$ gives three possibilities for $x_s$, namely $F_{m-3}$, $F_{m-4}$ and $F_{m-5}$. After manipulating (\ref{case3ai}), we find that these give $R(F_{m-5})$, $V(F_{m-5})$ and $w_{m-2}$ solutions respectively. Indeed, if $x_s=F_{m-5}$ then the system becomes 
\[
x_1+\cdots+x_{s-1}+F_{m-5}
=y_1+\cdots+y_{t-1}+F_{m-4} 
\le F_{m-3},
\]
but the right sum is guaranteed to exceed $F_{m-5}$ (recall again that $m \ge 7$ so $F_{m-4}>F_{m-5}$), so every solution to this corresponds uniquely to a solution of (\ref{case3a}) with $m$ replaced by $m-2$ and the roles of $x$ and $y$ switched. 

We now return to (\ref{case3a}) and count solutions for which 
$y_{t-1}=F_{m-5}$.
By \eqref{case3a} and \eqref{Wcons}, we have
\[
F_{m-5} 
< x_1 + \cdots + x_{s-1} + F_{m-4} 
= y_1 + \cdots + y_{t-2} + F_{m-5}
\le F_{m-2}
\]
and
\[
x_1<\cdots<x_{s-1}<F_{m-2},
\qquad
y_1<\cdots<y_{t-2}<F_{m-5}.
\]
If $x_{s-1} \ge F_{m-4}$ then the sum of the $x_i$ will be too large to equal the sum of the $y_j$. Moreover, the right sum is never greater than $F_{m-3}$. Thus, the number of solutions is simply $w_{m-2}$, whence
\begin{align*}
w_m &=
V(F_{m-3})-R(F_{m-3})
-(R(F_{m-5})+V(F_{m-5})+w_{m-2})
+ w_{m-2}
\\ &= 
V(F_{m-3})-R(F_{m-3})
-R(F_{m-5})-V(F_{m-5}).  
\end{align*}
Substituting this into (\ref{3total}), the total number of solutions in Case (3) is
\[
V(F_{m-1})-4V(F_{m-3})+2V(F_{m-5})
-2R(F_{m-1})+2R(F_{m-3})+2R(F_{m-5})
+1.
\]

\bigskip

Now we can move onto Case $(4)$. Notice that if $x_s=F_m$ and $y_t=F_{m-1}$, then $s=1$ and we can rearrange to find that there are $R(F_{m-2})$ solutions. Flipping $x_s$ and $y_t$ yields a total of $2R(F_{m-2})$ possibilities.

Lastly, for Case $(5)$, consider the original system (\ref{big}) with $x_s=F_{m-1}$ and $y_t=F_{m-2}$. The number of possibilities here is $w_{m+1}$ minus the number of solutions where the sums are in the range $\left(F_{m-2},F_{m-1}\right]$. However, if the sums are in this range then $s=1$ so, subtracting $F_{m-2}$ from both sides, we find that there are $R(F_{m-3})$ solutions. Hence, the total number of solutions in Case (5) is
\begin{align*}
&2(w_{m+1}-R(F_{m-3})) \\
&= 2(
V(F_{m-2}) - R(F_{m-2}) - R(F_{m-3})
- R(F_{m-4}) - V(F_{m-4})
).
\end{align*}

\bigskip

The total number of solutions across Cases (1), (2) and (4) is 
\[
V(F_{m-2}) + 2R(F_{m-2}).
\]
The total number of solutions across all cases is therefore
\begin{align*}
&V(F_m) - V(F_{m-1}) \\
&= V(F_{m-2}) + 2R(F_{m-2}) 
+ V(F_{m-1})-4V(F_{m-3})
\\ & \quad
+2V(F_{m-5})
-2R(F_{m-1})+2R(F_{m-3})+2R(F_{m-5}) +1 
\\ & \quad
+ 2(
V(F_{m-2}) - R(F_{m-2}) - R(F_{m-3})
- R(F_{m-4}) - V(F_{m-4})
),
\end{align*}
whence
\begin{align*}
V(F_m) &= 2V(F_{m-1}) + 3V(F_{m-2})
- 4V(F_{m-3}) - 2V(F_{m-4}) + 
2V(F_{m-5}) \\
&\qquad - 2R(F_{m - 1}) - 2R(F_{m - 4}) + 2R(F_{m - 5}) + 1.
\end{align*}

Finally, we insert Carlitz's \cite[Theorem 2]{Car1968} formula
\[
R(F_n) = \lfloor n/2 \rfloor
\qquad (n \ge 2).
\]
As $m \ge 7$, this gives
\begin{align*}
V(F_m) &= 2V(F_{m-1}) + 3V(F_{m-2})
- 4V(F_{m-3}) - 2V(F_{m-4}) + 
2V(F_{m-5}) \\
&\qquad - 
2\left \lfloor \frac{m-1}{2}
\right \rfloor
- 2 \left \lfloor
\frac{m-4}{2} \right \rfloor 
+ 2 \left \lfloor \frac{m-5}{2}
\right \rfloor + 1 \\
&= 2V(F_{m-1}) + 3V(F_{m-2})
- 4V(F_{m-3}) - 2V(F_{m-4}) + 
2V(F_{m-5}) \\
&\qquad + 1 - 2 \lfloor m/2 \rfloor,
\end{align*}
as claimed.
\end{proof}

\section{Solving the inhomogeneous linear recurrence}
\label{solve}

We write $v_m = V(F_m)$ for 
$m \ge 2$. Let
\[
\lam_1 \approx 2.48,
\qquad
\lam_2 \approx -1.17,
\qquad
\lam_3 = 1,
\qquad
\lam_4 = -1,
\qquad
\lam_5 \approx 0.69
\]
be the roots of the characteristic polynomial
\begin{align*}
\chi(x) 
&= x^5 - 2x^4 - 3x^3 + 4x^2 + 2x - 2 \\
&= (x-1)(x+1)(x^3-2x^2-2x+2).
\end{align*}
These roots generate a cubic field $K$. Further, note that $\lam_1$ is the same as it is in Theorem \ref{MainThm}. We wish to solve the inhomogeneous linear recurrence
\begin{equation}\label{vRec}
v_m = 2v_{m-1} + 3v_{m-2}
- 4v_{m-3} - 2v_{m-4} + 2v_{m-5} + 1 - 2 \lfloor m/2 \rfloor \qquad (m \ge 7)
\end{equation}
with initial data
\begin{equation}
\label{InitialData}
(v_2, v_3, v_4, v_5, v_6)
= (2, 3, 7, 12, 26).
\end{equation}

Let
\[
V = \begin{pmatrix}
\lam_1^2 & \lam_2^2 &
\lam_3^2 & \lam_4^2 & \lam_5^2 \\
\lam_1^3 & \lam_2^3 &
\lam_3^3 & \lam_4^3 & \lam_5^3 \\
\lam_1^4 & \lam_2^4 &
\lam_3^4 & \lam_4^4 & \lam_5^4 \\
\lam_1^5 & \lam_2^5 &
\lam_3^5 & \lam_4^5 & \lam_5^5 \\
\lam_1^6 & \lam_2^6 &
\lam_3^6 & \lam_4^6 & \lam_5^6
\end{pmatrix}.
\]
This matrix is invertible, as its determinant is $(\lam_1 \cdots \lam_5)^2$ times that of a Vandermonde matrix with pairwise distinct parameters. Define coefficients 
$c_1, \ldots, c_5 \in K$ by
\[
\begin{pmatrix}
c_1 \\
c_2 \\
c_3 \\
c_4 \\
c_5
\end{pmatrix}
= V^{-1} 
\begin{pmatrix}
1 \\
9/4 \\
3 \\
33/4 \\
17
\end{pmatrix} \approx
\begin{pmatrix}
0.0735 \\
-0.467
 \\
 0.625
 \\
 0.375
 \\
 0.394
\end{pmatrix}.
\]
For $m \ge 2$, we write
\[
\eps_m = m - 2 \lfloor
m/2 \rfloor
\]
for the remainder when $m$ is divided by two.

\begin{thm} For $m \ge 2$, we have
\[
v_m = \left(
\sum_{i \le 5} c_i \lam_i^m
\right) + \frac{m^2}{4}
- \frac{m \eps_m}{2}.
\]
\end{thm}

\begin{proof} We follow the standard approach \cite[Chapter 2]{GK1990}. The associated homogeneous recurrence is
\[
u_m = 2u_{m-1} + 3u_{m-2}
- 4u_{m-3} - 2u_{m-4} + 
2u_{m-5}, 
\]
which has characteristic polynomial $\chi$. The general solution to this is
\[
u_m = \sum_{i \le 5} C_i \lam_i^m,
\] 
where $C_1, \ldots, C_5$ are arbitrary constants.

\bigskip

Next, we search for a particular solution to (\ref{vRec}). The floor function is unwieldy, with $\lfloor{\frac{m}{2}}\rfloor$ being $\frac{m}{2}$ if $m$ is even but $\frac{m-1}{2}$ if $m$ is odd. This motivates looking for a particular solution which is defined piecewise depending on the parity of $m$. To this end, we seek functions $m \mapsto a_m$ and $m \mapsto b_m$ 
such that 
\[
a_m = 2b_{m-1} + 3a_{m-2} 
- 4b_{m-3}-2a_{m-4}+2b_{m-5}
+ 1 - m
\] 
and
\[
b_m = 2a_{m-1} + 3b_{m-2} 
- 4a_{m-3} - 2b_{m-4} 
+ 2a_{m-5} + 2 - m.
\]
If we can find two such functions, then a particular solution to \eqref{vRec} will be given by
\[
v_m = \begin{cases}
a_m, &\text{if } 2 \mid m \\
b_m, &\text{if } 2 \nmid m.
\end{cases}
\]

By comparing coefficients, we see that two linear polynomials will not work. We therefore try quadratic polynomials
\[
a_m = k_1 m^2 + k_2 m + k_3,
\qquad
b_m = \ell_1 m^2 + \ell_2 m 
+ \ell_3.
\]
Substituting these into our equations and equating coefficients, we see that this will work provided that 
\begin{align*}
k_2 &= 4k_1 + k_2 - 1, \\
k_3 &= 16\ell_1 - 20k_1 
+ 2k_2 + k_3 + 1
\end{align*}
and
\begin{align*}
\ell_2 &= 4\ell_1+ \ell_2-1, \\
\ell_3 &= 16k_1 - 20\ell_1 +2\ell_2 + \ell_3+2.
\end{align*} 
It is then immediate that 
\[
k_1 = \ell_1 = \frac{1}{4}, 
\qquad
k_2=0,
\qquad \ell_2 = -\frac{1}{2}
\]
works. We can also take $k_3 = \ell_3 = 0$, since there are no requirements on these. Consequently, a particular solution of (\ref{vRec}) is $v_m = \frac{m^2}{4}-\frac{m\varepsilon_m}{2}$. 

\bigskip

Adding this to our general homogeneous solution, we glean that the general solution to the inhomogeneous recurrence (\ref{vRec}) is given by
\[
v_m = 
\left( \sum_{i \le 5} C_i \lam_i^m \right) +
\frac{m^2}{4}-\frac{m\varepsilon_m}{2}.
\]
Finally, by definition, the $c_i$ satisfy the initial data \eqref{InitialData}, so the choice
\[
C_i = c_i \qquad (1 \le i \le 5)
\]
gives the result.
\end{proof}

As $\lam_1$ is the dominant root of $\chi$, i.e. the unique root with greatest absolute value, we obtain the following asymptotic formula for $V(F_m)$.

\begin{cor} We have
\[
V(F_m) \sim c_1 \lam_1^m
\qquad (m \to \infty).
\]
\end{cor}

As $F_{m+1} \asymp F_m \asymp \varphi^m$, and $V(\cdot)$ is monotonic, we thus obtain Theorem~ \ref{MainThm}.

\end{document}